\documentclass[a4paper,11pt]{article}
\usepackage{amsfonts}
\usepackage{amssymb}
\usepackage{amsmath}
\usepackage{amsthm}
\usepackage{verbatim}
\usepackage{algorithmic}
\usepackage[T1]{fontenc}
\usepackage{authblk}
\theoremstyle{plain}
\newtheorem{theorem}{Theorem}[section]
\newtheorem{lemma}[theorem]{Lemma}
\newtheorem{corollary}[theorem]{Corollary}
\newtheorem{proposition}[theorem]{Proposition}
\theoremstyle{definition}
\newtheorem{definition}[theorem]{Definition}
\newtheorem{remark}[theorem]{Remark}
\newtheorem{example}[theorem]{Example}

\usepackage{graphicx}
\usepackage{amsmath}
\usepackage{amssymb}

\usepackage{verbatim}
\usepackage{enumitem}
\usepackage{tikz}
\usetikzlibrary{matrix, positioning, decorations.pathreplacing, shapes.multipart}
\usepackage{xspace}

\DeclareMathOperator{\subf}{Sf}  
\DeclareMathOperator{\rf}{Rf} 

\DeclareMathOperator{\Prop}{\Phi} 
\newcommand{\X}{X}
\newcommand{\Y}{Y}

\DeclareMathOperator{\Lbl}{\Lambda} 

\newcommand{\op}[1]{\widehat{#1}} 
\newcommand{\opM}[1]{\widehat{\vphantom\wedge\smash{#1}}} 

\newcommand{\Logicname}[1]{\ensuremath{\mathsf{#1}}}
\newcommand{\GTS}{\Logicname{GTS}\xspace}
\newcommand{\FBS}{\Logicname{FBS}\xspace}
\newcommand{\CTL}{\Logicname{CTL}\xspace}
\newcommand{\ATL}{\Logicname{ATL}\xspace}
\newcommand{\LTL}{\Logicname{LTL}\xspace}

\newcommand{\card}{\ensuremath{\mathsf{card}\xspace}}
\newcommand{\tlimbound}{\Gamma} 

\begin{document}
\title{Bounded game-theoretic semantics for modal mu-calculus}
\author[1]{Lauri Hella}
\author[2]{Antti Kuusisto}
\author[3]{Raine R\"{o}nnholm}
\affil[1]{Tampere University, Finland}
\affil[2]{University of Helsinki and Tampere University, Finland}
\affil[3]{Universit\'{e} Paris-Saclay, France}

\date{}

\date{}

\maketitle

\begin{abstract}
\noindent
We introduce a new game-theoretic semantics (GTS) for the modal mu-calculus. Our so-called bounded GTS replaces parity games with alternative evaluation games where only finite paths arise; infinite paths are not needed even when the considered transition system is infinite. The novel games offer alternative approaches to various constructions in the framework of the mu-calculus. For example, they have already been successfully used as a basis for an approach leading to a natural formula size game for the logic. While our main focus is introducing the new GTS, we also consider some applications to demonstrate its uses. For example, we consider a natural model transformation procedure that reduces model checking games to checking a single, fixed formula in the constructed models, and we also use the GTS to identify new alternative variants of the mu-calculus with PTime model checking.
\end{abstract}
%


\section{Introduction}

The modal $\mu$-calculus \cite{Kozen83} is a well-known formalism that plays a
central role in, e.g., program verification.  The standard
semantics of the $\mu$-calculus is based on fixed points, but
the system has also a well-known game theoretic
semantics (\GTS) that makes use of parity games.
The related games generally involve infinite
plays, and the parity condition is used for determining
the winner (see, e.g., \cite{bradfield} for further details and a 
general introduction to the $\mu$-calclus).

\paragraph{The agenda and contributions of this article.}

In this article we present an alternative 
game-theoretic semantics for the modal $\mu$-calculus.
Our so-called \emph{bounded} \GTS is
based on games that resemble the standard semantic games for
the $\mu$-calculus, but there is an extra feature
that ensures that the plays within the novel framework
always end after a finite number of rounds.
Thereby only finite paths arise in related 
evaluation games \emph{even when investigating infinite transition systems}.
Thus there is no need to keep track of the parity
condition, so in that sense
the games we present in this article simplify the standard framework.
Furthermore, they offer an alternative perspective on the $\mu$-calculus, as we show that our
semantics is equivalent to the standard one. 
In the novel games, the evaluation of a fixed point formula begins
by one of the players declaring an ordinal number; the verifying player declares ordinals for $\mu$-formulae
and the falsifying player for $\nu$-formulae.
The declared ordinal is then lowered as the game proceeds, and
since ordinals are well-founded, the game will
indeed end in finite time, i.e.,
after a finite number of game steps.
In general, infinite ordinals are needed in the games, but
finite ordinals suffice in finite models.
While the bounded \GTS provides a new perspective on 
the standard modal {$\mu$-calculus}, our approach also leads 
naturally to a range of alternative semantic systems
that are not equivalent to
the standard semantics. Indeed, we divide the framework of bounded semantics into 
subsystems dubbed \emph{$\Gamma$-bounded semantics} for different ordinals $\Gamma$.
Here $\Gamma$ provides
a strict upper limit for the ordinals that can be used during the game play.
For each $\Gamma$-bounded semantics, we define also a compositional semantics and 
prove the game-theoretic and compositional versions equivalent.

If only finite ordinals are allowed, meaning $\Gamma = \omega$, we
obtain the \emph{finitely bounded} \GTS, which is an interesting system itself.
While this semantics is
equivalent to the standard case in finite models, the general expressive powers differ.
Indeed, we will show that the $\mu$-calculus under finitely bounded $\GTS$ does not
have the finite model property. As another interesting property, we
observe that the set of validities of the $\mu$-calculus under finitely bounded semantics is
strictly contained in the set of standard validities.

We then introduce yet another class of variants of the bounded \GTS consisting of the
systems of \emph{$f$-bounded semantics}. In the $\Gamma$-bounded
semantics, each $\mu$ and $\nu$-formula is associated with an ordinal of its own, while in
the $f$-bounded semantics this scheme is relaxed and only two ordinals are used, one for all $\mu$-formulae 
and another one for all $\nu$-formulae. The particular ordinals fixed in the beginning of the game
depend on the particular variant of $f$-bounded semantics. We prove PTime-completeness of 
the model checking problem of a range of simple yet expressive systems of $f$-semantics.
The result concerns both data and combined complexity. In addition to semantic
studies, we use \GTS to identify a 
canonical reduction of $\mu$-calculus model checking instances 
to checking a single,
uniform formula in the model obtained by the reduction. 

\paragraph{Further motivation of the study.}

While the formal results listed above are an important part of our study, the
focus of our article is mainly on the \emph{conceptual development} of the theory of
the $\mu$-calculus and related systems, not so much the more technical directions.
While some of the technical results we obtain have straighforward and \emph{obvious implicit similarities} to 
existing notions, such as finite approximants of fixed points, we
believe the systematic, formal and conceptual study
initiated in this article is justified.

Indeed, we believe the bounded \GTS in general can be a fruitful framework for various 
further developments.
The setting provides an alternative perspective to parity
games, replacing infinite plays with games based on
finitely many rounds only, thereby leading to a
conceptually interesting territory to be explored further. The fragments with PTime-model
checking we identify serve as an example of the various 
possibilities.
Furthermore, it is worth noting here, e.g., that the difference between
the standard and bounded \GTS for the $\mu$-calculus is
analogous to the relationship between while-loops and for-loops;
while-loops are iterated possibly infinitely long, whereas for-loops 
run for $k\in\mathbb{N}$ rounds, where $k$ can
generally be an input to the loop. Finally, we argue that the 
new semantics can quite often make formulae easier to read.

\paragraph{Notes on related work.}

There already exist several works where simple variants of
the bounded semantics have been considered in the context of
temporal logics with a significantly simpler recursion 
mechanisms than that of the $\mu$-calculus. The papers \cite{atl}, \cite{acmtrans2018}
consider a bounded semantics for the Alternating-time temporal logic \ATL, and \cite{atl2}, \cite{infcomp} extend the related study to
the extension $\ATL^+$ of \ATL.
See also \cite{time2017}, \cite{tcs2019}. Part of the original 
motivation behind the studies in \cite{atl}, \cite{acmtrans2018}, \cite{atl2}, \cite{infcomp}
(as well as the current article)
relates to work with the direct aim of understanding
variants and fragments of the general, expressively 
Turing-complete logic presented in \cite{turingcomplete}.
It is also worth mentioning that the work in the 
present article has 
already been made essential use of in 
constructing a canonical
formula size game for the $\mu$-calculus in \cite{sizegame}.

There is a whole range of earlier but closely
related logical studies that make use of notions 
with similar intuitions to the ones
behind the bounded semantics of this paper. 
Indeed,  for logics with time bounds, see, e.g.,
the paper \cite{toplas/AlurH98} on finitary fairness and the article \cite{Kupferman2007} relating to
promptness in Linear temporal logic \LTL. We also mention here the
work related to \emph{bounded model checking}, see, e.g., \cite{BiereCCZ99},
\cite{PenczekWZ02} and \cite{Zhang15}.
Finally, the article \cite{vardi} is one example of an early work that
uses explicit `clocking' of fixed point formulae in (a variant of)
the $\mu$-calculus, thereby involving some
ideas that bear a similarity to some
features used also in the present paper.





\section{Preliminaries}

\subsection{Syntax}

Let $\Prop$ be  a set of \emph{proposition symbols} and $\Lbl$ a set of \emph{label symbols}. 
Formulae of the modal $\mu$-calculus are generated by
the following grammar:
\begin{align*}
	\varphi ::= p \mid \neg p \mid \X \mid \varphi\vee\varphi \mid \varphi\wedge\varphi \mid \Diamond\varphi \mid \Box\varphi
			\mid \mu\X\varphi \mid \nu\X\varphi
\end{align*}
where $p\in\Prop$ and $\X\in\Lbl$.

%

Let $\varphi$ be a formula of the $\mu$-calculus.
The set of nodes in the syntax tree of $\varphi$ is denoted by $\subf(\varphi)$. All of these nodes correspond to some subformula of $\varphi$, but the same subformula may have several occurrences in the syntax tree of $\varphi$, as for example in the case of $p\vee p$. We  always distinguish between different occurrences of the same subformula, and thus we always assume that the position in the syntax tree of $\varphi$ is known for any
given subformula of $\varphi$. 
We also use the following notation:
\begin{multline*}
	\subf_{\mu\nu}(\varphi) := \\
	\{\theta\in\subf(\varphi) \mid \theta=\mu\X\psi \text{ or } \theta=\nu\X\psi \text{ for some }\psi\in\subf(\varphi) 
	\text{ and }X\in \Lambda\}.
\end{multline*}


\subsection{Standard compositional semantics}

\emph{A Kripke model} $\mathcal{M}$ is a tuple $(W, R, V)$, where $W$ is a nonempty set, $R$ a binary relation over $W$ and $V:\Phi\rightarrow\mathcal{P}(W)$ a valuation for proposition symbols in $\Prop$.
An \emph{assignment} $s:\Lbl\rightarrow\mathcal{P}(W)$ for $\mathcal{M}$ maps labels $\X$ to subsets of $W$.

\begin{definition}
Let $\mathcal{M}=(W,R,V)$ be a Kripke model, $w\in W$. Let $\varphi$ be a
formula of the $\mu$-calculus. 
\emph{Truth of $\varphi$ in $\mathcal{M}$ and $w$ under 
assignment $s$},
denoted by $\mathcal{M},w\vDash_s\varphi$, 
is defined as in standard modal logic for $p$, $\neg p$, $\vee$, $\wedge$, $\Diamond$, $\Box$.
%
%
%
%
The truth condition for label symbols is defined with respect to the assignment $s$:
\begin{itemize}
\item $\mathcal{M},w\vDash_s\X$ iff $w\in s(\X)$.
\end{itemize}
To deal with $\mu$ and $\nu$, we define an operator $\op{\varphi}_{\X,s}:\mathcal{P}(W)\rightarrow\mathcal{P}(W)$ such that
	$\op{\varphi}_{\X,s}(A) = \{w\in W\mid \mathcal{M},w\vDash_{s[A/\X]}\varphi\}$,
%
where $s[A/X]$ is the assignment that sends $X$ to $A$ and
treats other label symbols the same way as $s$.
%
The operators $\op{\varphi}_{\X,s}$ are always
monotone and thereby have least
and greatest fixed points.
%
\noindent
The semantics for the operators $\mu\X$ and $\nu\X$ is as follows:
\begin{itemize}
\item $\mathcal{M},w\vDash_s\mu\X\psi$ iff $w$ is in the least fixed point of the operator $\opM{\psi}_{\X,s}$.
\item $\mathcal{M},w\vDash_s\nu\X\psi$ iff $w$ is in the greatest fixed point of the operator $\opM{\psi}_{\X,s}$.
\end{itemize}
\end{definition}

A label symbol $\X$ is said to occur \emph{free} in a formula $\varphi$ if it occurs in $\varphi$ but is not a subformula of any subformula of $\varphi$ of the form $\mu\X\psi$ or $\nu\X\psi$. A formula $\varphi$ is
called a \emph{sentence} if it does not contain any free label symbols. 
Truth of a
sentence $\varphi$ is independent of assignments $s$, so 
we may simply write $\mathcal{M},w\vDash\varphi$
instead of $\mathcal{M},w\vDash_s\varphi$.


\subsection{Alternating reachability
games}\label{alternatingreachabilitysection}

The \emph{alternating reachability game} problem is 
defined as follows.
The input to the problem is a finite
pointed model $(\mathcal{M},w)$, i.e., $\mathcal{M}$ is 
Kripke model and $w$ a state in it. We assume the
vocabulary of $\mathcal{M}$ contains the 
proposition symbols $p_B$ and $q_B$.
The game is played by two
players, $A$ and $B$, starting from
the state $w$. In each round, one of the players moves (if possible) to 
some state that can be
directly reached in one step from the current state via the accessibility 
relation; if $q_{B}$
holds in the current state, then $B$ moves, and
otherwise $A$ moves. 
If the players 
reach a state where $p_{B}$ holds,
then the game ends and $\mathrm{B}$ wins. 
If a player cannot make the required
move in some state (meaning the state is a dead end), then the game ends 
and that player loses and the other
player wins. If the game does not end in a
finite number of moves, then $\mathrm{A}$ wins.
The alternating reachability game problem yields the
answer \emph{yes} on the input $(\mathcal{M},w)$ 
iff $\mathrm{B}$ has a winning strategy in the game.

It is clear the game can be played also on infinite models.
We let $\mathcal{AR}$ denote the class of all positive instances of
the alternating reachability game problem, including infinite models.
The corresponding class of
finite instances is $\mathcal{AR}_{\mathit{fin}}$.
The following observation is well known.

\begin{proposition}\label{alternatingreachabilityproposition}
Let $\mathcal{M}$ be a Kripke model with one 
accessibility relation and propositional
vocabulary $\{p_\mathrm{B},q_B\}$. Let $w$ be a state in $\mathcal{M}$.
Then we have
$
(\mathcal{M},w)\in\mathcal{AR}$ iff $\mathcal{M},w\vDash\chi,
$
where $\chi := 
\mu X\bigl(p_B\vee (q_B \wedge \Diamond X)\vee(\neg q_{B}
\wedge \Box X)\bigr)$.
\end{proposition}

\section{Bounded game-theoretic semantics}

Recall that the general 
idea of game-theoretic semantics (\GTS) is that
truth of a formula $\varphi$ is 
checked in a model $\mathcal{M}$ via playing a game where a proponent
player (Eloise) attempts to show that $\varphi$ holds in $\mathcal{M}$
while an opponent player (Abelard) tries to
establish the opposite---that $\varphi$ is false.
In this section we define a \emph{bounded game-theoretic semantics} for
the $\mu$-calculus, or bounded \GTS. The semantics shares some features 
with a similar \GTS for the \emph{Alternating-time temporal logic} (\ATL)
defined in \cite{atl} (see also \cite{atl2}). 

\subsection{Bounded evaluation games}

Let $\varphi$ be a sentence of the $\mu$-calculus and $\X\in\subf(\varphi)$. The \emph{reference formula of $\X$},
denoted $\rf(\X)$, is the unique
subformula of $\varphi$ that \emph{binds} $\X$.
That is, $\rf(\X)$ is of the form $\mu\X\psi$ or $\nu\X\psi$
and there is no other operator $\mu X$ or $\nu X$
in the syntax tree on the path from $\rf(\X)$ to $\X$.
%
Since $\varphi$ is a
sentence, every label symbol has a reference formula (and the reference formula is by definition unique for each label symbol).

\begin{example}\label{ex: Sentence}
Consider the sentence 
    $\varphi^* := \nu X\Box\mu Y\bigl(\Diamond Y \vee (p\wedge X)\bigr)$.
Here we have $\rf(X)=\varphi^*$ and $\rf(Y)=\mu Y(\Diamond Y \vee (p\wedge X))$.
\end{example}

\begin{definition}
Let $\mathcal{M}$ be a model, $w_0\in W$, $\varphi_0$ a sentence and $\tlimbound>0$ an ordinal. We define the \emph{$\tlimbound$-bounded evaluation game} $\mathcal{G}=(\mathcal{M},w_0,\varphi_0,\tlimbound)$ as follows.
The game has two players, \emph{Abelard} and \emph{Eloise}. The \emph{positions} of the game are of the form $(w,\varphi,c)$, where $w\in W$, $\varphi\in\subf(\varphi_0)$ and
	$c:\subf_{\mu\nu}(\varphi_0)\;\rightarrow\; \{\gamma\mid\gamma\leq\tlimbound\}$
is a \emph{clock mapping}. We call the value $c(\theta)$ the \emph{clock value of} $\theta$ (for $\theta\in\subf_{\mu\nu}(\varphi_0)$).

The game begins from the \emph{initial position} $(w_0,\varphi_0,c_0)$, where $c_0(\theta)=\tlimbound$ for every $\theta\in\subf_{\mu\nu}(\varphi_0)$. The game is then played according to the following rules:
\begin{itemize}
\item In a position $(w,p,c)$ for some $p\in\Prop$, Eloise wins the game if $w\in V(p)$. Otherwise Abelard wins the game.
\item In a position $(w,\neg p,c)$ for some $p\in\Prop$, Eloise wins the game if $w\notin V(p)$. Otherwise Abelard wins the game.
\item In a position $(w,\psi\vee\theta,c)$, Eloise selects whether the next position of the game is $(w,\psi,c)$ or $(w,\theta,c)$.
\item In a position $(w,\psi\wedge\theta,c)$, Abelard selects whether the next position of the game is $(w,\psi,c)$ or $(w,\theta,c)$.
\item In a position $(w,\Diamond\psi,c)$,
Eloise selects some $v\in W$ s.t. $wRv$ and the
next position is $(v,\psi,c)$. If
there is no such $v$, then Abelard wins the game.
\item In a position $(w,\Box\psi,c)$, Abelard selects some $v\in W$ s.t. $wRv$ and the next position is $(v,\psi,c)$. If there is no such $v$, then Eloise wins the game.
\item In a position $(w,\mu\X\psi,c)$, Eloise chooses an ordinal
$\gamma<\tlimbound$. Then the game continues from the position $(w,\psi,c[\gamma/\mu\X\psi])$. Here $c[\gamma/\mu\X\psi]$ is
the clock mapping that sends $\mu\X\psi$ to $\gamma$ and treats other formulae as $c$.
\item In a position $(w,\nu\X\psi,c)$, Abelard chooses an ordinal $\gamma<\tlimbound$. Then the game continues from the position $(w,\psi,c[\gamma/\nu\X\psi])$.

\item Suppose that the game is in a position $(w,\X,c)$. 
Let $c(\rf(\X))=\gamma$.
\begin{enumerate}
\item Suppose that $\rf(\X)=\mu\X\psi$ for some $\psi$.
\begin{itemize}
\item If $\gamma=0$, then Abelard wins the game.
\item Else, Eloise must select some $\gamma'<\gamma$, and the game continues from the position $(w,\psi,c')$, where 
\begin{itemize}
\item $c'(\mu\X\psi)=\gamma'$,
\item $c'(\theta)=\Gamma$ 
for all $\theta\in\subf_{\mu\nu}(\varphi_0)$ s.t. $\theta\in\subf(\psi)$,
\item $c'(\theta)=c(\theta)$ for all other $\theta\in\subf_{\mu\nu}(\varphi_0)$.
\end{itemize}
\end{itemize}
\item Suppose that $\rf(\X)=\nu\X\psi$ for some $\psi$.
\begin{itemize}
\item If $\gamma=0$, then Eloise wins the game.
\item Else, Abelard must select some $\gamma'<\gamma$, and the game continues from the position $(w,\psi,c')$, where 
\begin{itemize}
\item $c'(\nu\X\psi)=\gamma'$,
\item $c'(\theta)=\Gamma$ 
for all $\theta\in\subf_{\mu\nu}(\varphi_0)$ s.t. $\theta\in\subf(\psi)$,
\item $c'(\theta)=c(\theta)$ for all other $\theta\in\subf_{\mu\nu}(\varphi_0)$.
\end{itemize}
\end{itemize}
\end{enumerate}
\end{itemize}
The positions  where one of the players wins the game are called \emph{ending positions}.
The execution of the rules related to a position of the game
constitutes one \emph{round} of the game. The number of
rounds in a play of the game is called the \emph{length of the play}.
We call the ordinals $\gamma<\tlimbound$ \emph{clock values}
and the ordinal $\tlimbound$ the \emph{clock value bound}.
(We note that only rounds with formulae of type $\mu X\psi$, $\nu X\psi$
and $X$ affect clock values.)

%
%
%
\end{definition}

Now, we observe that in \GTS, we have no need for assignments $s$.
A label symbol in $X\in\Lambda$ is simply a marker that
points to a node (that node being the formula $\rf(X)$) in the syntax tree of the sentence $\varphi_0$. Hence label symbols are conceptually quite different in \GTS and compositional semantics.
Indeed, the operators $\mu X$ (respectively $\nu X$) can be
given a natural reading relating to \emph{self-reference}. In the
formula $\mu X\psi$, the operator $\mu X$ is \emph{naming} the
formula $\psi$ with the name $X$.
The atoms $X$ inside $\psi$ are, in turn,
\emph{claiming} that $\psi$ holds, i.e., referring back to
the formula $\psi$.
The difference between $\mu$ and $\nu$ is that $\mu X \psi$
relates to \emph{verifying} the
formula $\psi$ while $\nu X \psi$ is
associated with preventing the
falsification of $\psi$, i.e., \emph{defending} $\psi$.
Therefore, if $N(\psi)$ denotes a
natural language reading of $\psi$, then the natural language 
reading of $\mu X\psi$ states that ``\emph{we can verify the claim
named $X$ which asserts that $N(\psi)$}''. An analogous
reading can be given to $\nu X\psi$. This scheme of reading  
recursive formulae
via self-reference is from \cite{turingcomplete}, \cite{final}.

\begin{example}\label{ex: Game}
Consider the Kripke model $\mathcal{M}^*=(W,R,V)$, where we have $W=\{w_i\mid i\in\mathbb{N}\}$, $R=\{(w_0,w_i)\mid i\geq 1\}\cup\{(w_{i+1},w_i)\mid i\geq 0\}$ and $V(p)=\{w_0\}$.


\begin{center}
\begin{tikzpicture}[
	scale=1.1,
	state/.style={draw, circle, rounded corners, font=\small}
	]
	\node at (0,1.3) {$\mathcal{M}^*$:};
	\node at (3,1.3) [state] (0) {$p$};
	\node at (1,0) [state] (1) {\phantom{$p$}};
	\node at (2,0) [state] (2) {\phantom{$p$}};
	\node at (3,0) [state] (3) {\phantom{$p$}};
	\node at (4,0) [state] (4) {\phantom{$p$}};
	\node at (5,0) (5) {};
	\node at (5.5,0) {\dots};
	\node at (5,0.7) {\dots};
	\draw[-latex] (0) to (1);
	\draw[-latex] (0) to (2);
	\draw[-latex] (0) to (3);
	\draw[-latex] (0) to (4);
    \draw[-latex] (0) to (4.8,0.3);
	\draw[-latex] (2) to (1);
	\draw[-latex] (3) to (2);	
	\draw[-latex] (4) to (3);
	\draw[-latex] (5) to (4);
	\draw[-latex, bend left] (1) to (0);
	\node[right=3pt of 0] {$w_0$};
	\node[below=3pt of 1] {$w_1$};
	\node[below=3pt of 2] {$w_2$};
	\node[below=3pt of 3] {$w_3$};
	\node[below=3pt of 4] {$w_4$};
\end{tikzpicture}
\end{center}

Recall the sentence $\varphi^*=\nu X\Box\mu Y(\Diamond Y \vee (p\wedge X))$ from Example~\ref{ex: Sentence} and consider the evaluation game $\mathcal{G}^*=(\mathcal{M}^*,w_0,\varphi^*,\omega)$.
In $\mathcal{G}^*$, Abelard first announces a clock value $n<\omega$ for $\rf(X)$ and then makes a jump from the intial state $w_0$ (with a $\Box$-move). 
Next Eloise announces some clock value $m<\omega$ for $\rf(Y)$. Then
she can, by repeated $\vee$-moves, 
jump in the model (making a $\Diamond$-move) and loop back to the
formula $\rf(Y)$; each time she loops back, she
needs to lower the value of $m$. If Eloise at 
some point chooses the right disjunct, Abelard can
either check if $p$ true in the current state or
loop back to $\rf(X)$. In the latter
case, the value of $n$ is lowered, but the
value of $m$ is reset
back to $\omega$ (allowing Eloise to
choose a fresh value $m$ next time).

The game eventually ends when (1) the clock value of $\rf(X)$ goes to zero, whence Abelard loses; when (2) the clock value of $\rf(Y)$ goes to zero, whence Eloise loses; or when (3) Abelard chooses the left conjunct, whence Eloise wins iff $p$ is true at the current state.
%
%
We will return to this game in Example~\ref{ex: Strategy}.
\end{example}

\begin{proposition}\label{the: finite games}
Let $\mathcal{G}=(\mathcal{M},w,\varphi,\tlimbound)$ be a bounded evaluation game. Every play of $\mathcal{G}$ ends in a finite
number of rounds.
\end{proposition}

\begin{proof}
For each positive integer $k$, 
let $\prec_k$ denote the ``canonical lexicographic order'' of
$k$-tuples of ordinals. 
That is, $(\gamma_1,\dots,\gamma_k)\prec_k (\gamma_1',\dots,\gamma_k')$ iff
there exists some $i\leq k$ such that $\gamma_i<\gamma_i'$
and $\gamma_j = \gamma_j'$ for all $j<i$.
Consider a branch in the syntax tree of $\varphi$. 
Let $\psi_1,\dots ,\psi_k\in\subf_{\mu\nu}(\varphi)$ be the $\mu\nu$-formulae occurring on this branch in this order 
(starting from the root).
In each round of the game, each such sequence $(\psi_1,\dots ,\psi_k)$
is associated with the $k$-tuple $(c(\psi_1),\dots,c(\psi_k))$ of clock values 
(that are ordinals less or equal to $\Gamma$).
It is easy to see that if $c$ and $c'$ are clock mappings such that $c'$
occurs later than $c$ in
the game, then we have $(c'(\psi_1),\dots,c'(\psi_k))\preceq_k(c(\psi_1),\dots,c(\psi_k))$.
Also, every time a transition from some label $X$ to 
the reference formula $\rf(X)$ is made, there is at least one branch
where the $k$-tuple (for the relevant $k$) of
clock values becomes strictly lowered (in
relation to $\prec_k$).
As ordinals are well-founded, it is clear
that the game must end after finitely many rounds.
%
\end{proof}

Each evaluation game $\mathcal{G}$ can naturally be associated with a \emph{game tree}
$T(\mathcal{G})=(P_\mathcal{G},E_\mathcal{G})$,
where $P_\mathcal{G}$ is the set of positions $(v,\psi,c)$ of $\mathcal{G}$ 
and $E_\mathcal{G}$ is the successor position relation.
$T(\mathcal{G})$ is formed by
beginning from the initial position and adding
transitions to all possible successor positions.
This procedure is then repeated from the
successor positions until an ending
position is reached. In the game tree, the initial
position is of course the root and ending
positions are leafs. Complete branches correspond to possible plays of the game. 
Due to Proposition~\ref{the: finite games}, the game tree of any
bounded evaluation game is well-founded, i.e., it does not
contain infinite branches. However, if the clock value bound $\tlimbound$ is infinite,
then the out-degree of some of the nodes of the
game tree can be infinite.
%
%
%

%

\subsection{Game-theoretic semantics}


\begin{definition}
Let $\mathcal{G}=(\mathcal{M},w_0,\varphi_0,\tlimbound)$ be an evaluation game. A \emph{strategy} $\sigma$ for Eloise in $\mathcal{G}$ is a partial mapping on the set of those positions $(w,\varphi,c)$ of the game where Eloise needs to make a move such that: 
$\sigma(w,\psi\vee\theta,c)\in\{\psi,\theta\}$,
$\sigma(w,\Diamond\psi,c)\in\{v\in W \mid wRv\}$,
$\sigma(w,\mu\X\psi,c)\in\{\gamma \mid \gamma<\tlimbound\}$,
and $\sigma(w,X,c)\in\{\gamma \mid \gamma<c(\rf(\X))\}$
where $\rf(\X)$ is of the form $\mu\X\psi$.
%
%
%
%
%
%
%
%
%
We say that Eloise plays according to $\sigma$ if she makes all her choices according to 
$\sigma$
%
%
%
and that $\sigma$ is a \emph{winning strategy} if Eloise always wins when playing according to $\sigma$. 
\end{definition}



%
%


\begin{definition}
Let $\mathcal{M}=(W,R,V)$ be a
model, $w\in W$, $\varphi$ a sentence and $\tlimbound>0$ an ordinal. 
We define truth of $\varphi$ in $\mathcal{M}$
and $w$ according to \emph{$\tlimbound$-bounded game theoretic semantics}, $\mathcal{M},w\Vdash^\tlimbound\!\varphi$, as follows:
\[
	\mathcal{M},w\Vdash^\tlimbound\!\varphi 
	\text{\; iff \; Eloise has a winning strategy in } (\mathcal{M},w,\varphi,\tlimbound).
\]
\end{definition}

\begin{example}\label{ex: Strategy}
Recall the game $\mathcal{G}^*$ from Example~\ref{ex: Game}. 
We define a strategy for Eloise as follows. After Abelard has made a transition to some state $w_j$, Eloise chooses $j$ for the clock value of $\rf(Y)$ and jumps in the model until reaching again $w_0$. She chooses the right disjunct at $w_0$, whence she either wins (since $w_0\in V(p)$) or Abelard needs to lower the clock value of $\rf(X)$ and the clock value of $\rf(Y)$ gets reset back to $\omega$.
Clearly this is a winning strategy Eloise and thus $\mathcal{M}^*,w_0\Vdash^\omega\!\varphi^*$.

From the structure of the evaluation games for $\varphi^*$ we find an interpretation for the meaning of $\varphi^*$: ``we can infinitely repeat the process where first (1) an arbitrary transition is made, and then (2) we can reach a state where $p$ is true and the process can be continued from (1)''.
Hence the clock value chosen for $\rf(Y)$ is
intuitively a ``commitment'' on
\emph{how many rounds at
most it will take to reach a state where $p$ holds}.
The clock value for $\rf(X)$, on the
other hand, is a ``challenge'' on
\emph{how many times $p$ must be reached}.
Indeed, in models where $p$ can be
reached only finitely many---say $n$---times
from the initial state, Abelard can win by
choosing $n+1$ as the
initial clock value for $\rf(X)$.


\end{example}

\section{Bounded compositional semantics}

In this section we define a compositinal semantics based on 
\emph{ordinal approximants} of fixed point operators.
Let $\mathcal{M}=(W,R,V)$ be a Kripke model,  $F:\mathcal{P}(W)\rightarrow\mathcal{P}(W)$ an operator and $\gamma$ an ordinal. We define the sets $F_\mu^\gamma$ and $F_\nu^\gamma$  recursively as follows:
\begin{align*}
	&F_\mu^0 := \emptyset 
	    &\text{ and } \quad &F_\nu^0 := W. \\
	&F_\mu^\gamma := F\bigl(F_\mu^{\gamma-1}\bigr)
    	&\text{ and } \quad &F_\nu^\gamma := F\bigl(F_\nu^{\gamma-1}\bigr),
		\;\text{ if $\gamma$ is a successor ordinal}. \\
	&F_\mu^\gamma := \bigcup_{\delta<\gamma}F_\mu^\delta
        &\text{ and } \quad &F_\nu^\gamma := \bigcap_{\delta<\gamma}F_\nu^\delta,
        \;\;\quad\text{ if $\gamma$ is a limit ordinal}.
\end{align*}

\begin{definition}
Consider a model $\mathcal{M}$ with a state $w$ and a
related assignment $s$.
We obtain \emph{$\tlimbound$-bounded
compositional semantics}
for the $\mu$-calculus by
defining truth of $p$, $\neg p$, $\vee$, $\wedge$, $\Diamond$, $\Box$ and $\X$ recursively as in 
the standard compositional semantics and
treating the $\mu$ and $\nu$-operators as follows:
\begin{itemize}
\item $\mathcal{M},w\vDash^\tlimbound_s\mu\X\psi$ iff $w\in(\opM{\psi}_{\X,s,\tlimbound})_\mu^\tlimbound$,
\item $\mathcal{M},w\vDash^\tlimbound_s\nu\X\psi$ iff $w\in(\opM{\psi}_{\X,s,\tlimbound})_\nu^\tlimbound$,
\end{itemize}
where the operator $\op{\varphi}_{\X,s,\tlimbound}:\mathcal{P}(W)\rightarrow\mathcal{P}(W)$ is defined such that
\begin{align*}
	\op{\varphi}_{\X,s,\tlimbound}(A) = \{w\in W\mid \mathcal{M},w\vDash^\tlimbound_{s[A/\X]}\varphi\}.
\end{align*}
\end{definition}

%
%

\noindent
The semantics of the $\mu$ and $\nu$-operators can be equivalently given as follows:
\begin{itemize}
\item $\mathcal{M},w\vDash^\tlimbound_s\mu\X\psi$ iff there exists some $\gamma<\tlimbound$ s.t. $w\in(\opM{\psi}_{\X,s,\tlimbound})_\mu^{\gamma+1}$.
\item $\mathcal{M},w\vDash^\tlimbound_s\nu\X\psi$ iff $w\in(\opM{\psi}_{\X,s,\tlimbound})_\nu^{\gamma+1}$ for every $\gamma<\tlimbound$.
\end{itemize}
If $\tlimbound$ is a limit ordinal, we can replace the superscripts $\gamma+1$ above by $\gamma$.

%
%

\medskip

We say that a formula is in \emph{normal form} if 
each label symbol in $\Lambda$ occurs in the 
formula at most once in the $\mu$-$\nu$-operators
(but may occur several times on the atomic level).
We let $\varphi'$ denote a normal form variant of $\varphi$
obtained simply by renaming label symbols where
appropriate.
It is easy to show that $\varphi$ is equivalent to $\varphi'$ with respect to both $\tlimbound$-bounded compositional semantics ($\vDash^\tlimbound$) and $\tlimbound$-bounded \GTS ($\Vdash^\tlimbound$). Therefore, when proving the equivalence of these two semantics, it suffices that we consider sentences that are in normal form.
Indeed, henceforth we assume that all formulae are in this normal form.

\begin{theorem}\label{the: Equivalence of semantics}
Let $\tlimbound$ be an ordinal, $\mathcal{M}$ a Kripke model, $w_0\in W$ and $\varphi_0$ a sentence of the $\mu$-calculus. Now we have
\[
	\mathcal{M},w_0\vDash^\tlimbound\!\varphi_0 \;\text{ iff }\;  \mathcal{M},w_0\Vdash^\tlimbound\!\varphi_0. 
\]
\end{theorem}

\begin{proof}(Sketch.) 
We present here a proof sketch
highlighting the main ideas. For a fully detailed proof, see the appendix.
The key in both directions 
of the proof
is the following condition $(\star)$ which is a property
satisfied/unsatisfied by positions $(w,\varphi,c)$ in the evaluation game $\mathcal{G}=(\mathcal{M},w_0,\varphi_0,\tlimbound)$:
\begin{itemize}
\item[($\star$)]
There is an assignment $s$ such that $\mathcal{M},w\vDash_s^{\Gamma}\varphi$, and
for each $\X\in\subf(\varphi_0)$:
\begin{enumerate}
\item
$s(X) = (\opM{\psi}_{\X,s,\tlimbound})_\mu^{\gamma}$ if $c(\rf(\X)) = \gamma$
and $\rf(\X) = \mu\X\psi$,
\item
$s(X) = (\opM{\psi}_{\X,s,\tlimbound})_\nu^{\gamma}$ if $c(\rf(\X)) = \gamma$
and $\rf(\X) = \nu\X\psi$.
\end{enumerate}
\end{itemize}
Note that this condition essentially relates the clock values $\gamma$ of bounded \GTS to $\gamma$-approximants in the bounded compositional semantics.

Proving the left to right implication, we first note that $(\star)$ holds in the initial position of $\mathcal{G}$ by the assumption $\mathcal{M},w_0\vDash^\tlimbound\!\varphi_0$. 
Then we show that whenever $(\star)$ holds for the current position, Eloise either wins the game in the current position or she can
maintain $(\star)$ to the next position. By maintaining $(\star)$, we obtain a winning strategy since $\mathcal{G}$ ends in a finite number of rounds.

For the other direction of the equivalence, we suppose that Eloise has a winning strategy $\sigma$ in $\mathcal{G}$. Since the game tree of $\mathcal{G}$ is well-founded, we can use well-founded (backwards) induction on the positions in the tree to prove that: if a $(w,\varphi,c)$ in can be reached with $\sigma$, then ($\star$) holds for $(w,\varphi,c)$.
Hence, in particular, $(\star)$ holds in the initial position of $\mathcal{G}$ and thus $\mathcal{M},w_0\vDash^\tlimbound\!\varphi_0$.
%
\end{proof}



Let $\mathcal{M}$ be a model. It is well-known that over  $\mathcal{M}$, each operator related to a formula of the $\mu$-calculus reaches a fixed point in at most $(\card(\mathcal{M}))^+$ iterations, where $(\card(\mathcal{M}))^+$ is the 
successor \emph{cardinal} of $\card(\mathcal{M})$. Thus it is
easy to see that the standard compositional semantics and $(\card(\mathcal{M}))^+$-bounded compositional semantics are
equivalent in $\mathcal{M}$. Hence obtain the following corollary:

\begin{corollary}\label{the: bounded vs standard semantics}
$\tlimbound$-bounded $\GTS$ is equivalent with the standard compositional semantics of the $\mu$-calculus when $\tlimbound\geq (\card(\mathcal{M}))^+$.
\end{corollary}

Also note that, in the special case of \emph{finite models}, it suffices to use finite clock values that are at most the cardinality of the model.


\section{Finitely bounded semantics}

As stated in Corollary~\ref{the: bounded vs standard semantics}, the bounded semantics becomes equivalent with the standard (unbounded) semantics if we set a sufficiently large clock value bound $\tlimbound$. However, using smaller values of $\tlimbound$, we obtain different semantic systems typically nonequivalent to the standard semantics.
We can either set some fixed bound for $\tlimbound$ or use a value that is determined by some parameters---such as the size of the given model and the given formula. In this section we consider the former case; systems
relating to the latter case is examined in Section~\ref{Variants with PTime model checking}.


A particularly interesting case with a fixed value of $\tlimbound$ is the so-called \emph{finitely bounded semantics} where we set $\tlimbound=\omega$ for all evaluation games. In the corresponding \GTS, the players can only announce \emph{finite} clock values.\footnote{Note that the correspondence to for-loops is particularly natural with finitely bounded semantics: iterations can be done up to any finite bound that has to be declared in advance.} 
Finitely bounded semantics will be denoted by \FBS which refers to both game-theoretic and compositional semantics with $\tlimbound=\omega$. 
%
%
In finite models, \FBS is equivalent to the standard semantics, but this equivalence breaks over infinite models; see Example~\ref{ex: FBS} below.

In the example and proofs that follow, we will consider the sentence 
\[
    \varphi_{\mathrm{AF}p}:=\mu\X(p\vee\Box X)
\]
which intuitively means that on every path, $p$ can be reached eventually. 
Note that $\varphi_{\mathrm{AF}p}$ corresponds to the sentence $\mathrm{AF}p$ of \emph{Computation tree logic} \CTL.

\begin{example}\label{ex: FBS}
Recall the model $\mathcal{M}^*$ from Example~\ref{ex: Game}.
Let $\mathcal{M}^\dagger$ be the model that is otherwise identical to $\mathcal{M}^*$, but $V(p)=\{w_1\}$.
Since the state $w_1$ is eventually reached on every path starting from $w_0$, it is easy to see that $\mathcal{M}^\dagger,w_0\models \varphi_{\mathrm{AF}p}$. However, $\mathcal{M}^\dagger,w_0\not\models^\omega \varphi_{\mathrm{AF}p}$ since from $w_0$ there is no finite upper bound on how many transitions 
are needed
to reach $w_1$. Indeed, Abelard has a winning strategy in $(\mathcal{M}^\dagger,w_0,\varphi_{\mathrm{AF}p},\omega)$ since he can win by choosing a transition to $w_{j+1}$ for any clock value $j<\omega$ for $\rf(\X)$---chosen by Eloise.

It is worth noting that $\mathcal{M}^\dagger,w_0\models^{\omega+1} \varphi_{\mathrm{AF}p}$ since if Eloise can choose $\omega$ as the initial clock value for $\rf(\X)$ and then lower it to $j-1$ after Abelard has made a transition to a state $w_j$.
Moreover, we also have $\mathcal{M}^\dagger,w_0\models^\omega \Box\varphi_{\mathrm{AF}p}$ since Eloise will know how many transitions it takes to reach $w_1$ as Abelard has to make the first transition before Eloise must announce a clock value.
\end{example}

In the proofs that follow, we will use negations and
implications of $\mu$-calculus formulae. Such formulae are in
general not included in our official
syntax (in the current paper), but it is straightforward to show
that they can be translated to equivalent formulae in negation normal form.

It is well known that, with standard semantics, the $\mu$-calculus has the \emph{finite model property}, i.e., every satisfiable sentence is satisfied in some finite model (see, e.g., \cite{bradfield}). However, with \FBS this property is lost.

\begin{proposition}\label{the: finite model property}
Modal $\mu$-calculus with finitely bounded semantics does not have the finite model property.
\end{proposition}

\begin{proof}
It is easy to see that $\Box\varphi_{\mathrm{AF}p}\rightarrow \varphi_{\mathrm{AF}p}$ is valid with the standard semantics (this follows from the ``fixpoint property'' $\mathrm{AF}p\leftrightarrow p\vee \mathrm{AXAF}p$ of \CTL). Therefore $\Box\varphi_{\mathrm{AF}p}\wedge\neg \varphi_{\mathrm{AF}p}$ is not satisfiable with the standard semantics. As the standard semantics is equivalent to \FBS in
finite models, $\Box\varphi_{\mathrm{AF}p}\wedge\neg \varphi_{\mathrm{AF}p}$ cannot be satisfied under \FBS in any finite model.
However, $\Box\varphi_{\mathrm{AF}p}\wedge\neg \varphi_{\mathrm{AF}p}$ is satisfiable with \FBS in an infinite model---as demonstrated in Example~\ref{ex: FBS}.
\end{proof}

Moreover, \FBS has the following interesting connection to the standard semantics in relation to valid sentences.

\begin{proposition}\label{the: validities}
The set of validities of the $\mu$-calculus with
finitely bounded semantics is 
strictly included in the set of validities with the standard semantics. 
\end{proposition}

\begin{proof}
For proving the inclusion, let $\varphi$ be a sentence valid under \FBS. Then, in particular, $\neg\varphi$ cannot be satisfied under \FBS in any finite model. Since the standard semantics and \FBS are equivalent in finite models, it follows that $\neg\varphi$ is not satisfied by the standard semantics in any finite model. Due to the finite model property of the standard semantics, $\neg\varphi$ is not satisfied by any model and thus $\varphi$ is valid.
The inclusion is strict as $\Box\varphi_{\mathrm{AF}p}\rightarrow \varphi_{\mathrm{AF}p}$ is valid with the standard semantics but not with \FBS (cf. proof of
Proposition \ref{the: finite model property}).
\end{proof}

We showed in \cite{time2017}, \cite{tcs2019} that the claims of Propositions \ref{the: finite model property}, \ref{the: validities} hold also for the \FBS defined for \CTL and \ATL.
There we also developed a tableau method for showing that the validity problem of \CTL and \ATL with \FBS is decidable and has the same complexity (ExpTime) as with the standard semantics.


\section{Variants with PTime model checking}\label{Variants with PTime model checking}

The bounded $\GTS$ leads naturally to semantic variants of the $\mu$-calculus 
that can quite directly be shown to have PTime complete model checking. The main
point is to make use of the intimate relationship between
alternating Turing machines and semantic games. The novel systems of
semantics we consider resemble the $\Gamma$-bounded semantics but
utilize a simplified way to control
how many times $\mu$ and $\nu$-formulae can be repeated in semantic games.

To present the alternative 
semantic systems in detail, let $f$ be a map
that takes as input a model $\mathcal{M}$, a
point $w$ in the domain $W$ of $\mathcal{M}$ and a
sentence $\varphi$, outputting an ordinal.
We  assume that if $g$ is an isomorphism 
from $\mathcal{M}$ to $\mathcal{M}'$, then $f(\mathcal{M},w,\varphi)
= f(\mathcal{M}',f(w),\varphi)$.\footnote{We note that $f$ is too large to be a set, but
this is unproblematic to our study.} The function $f$ gives
rise to the \emph{simple $f$-bounded} \GTS
defined as follows.

\begin{definition}
Let $\mathcal{M}$ be a Kripke-model, $w\in W$ and $\varphi$ a sentence of the $\mu$-calculus.
The \emph{simple $f$-bounded evaluation game} $\mathcal{G}_f=(\mathcal{M},w,\varphi)$ is
played the same way as the $\Gamma$-bounded
evaluation game $\mathcal{G}_{\Gamma} =(\mathcal{M},w,\varphi,\Gamma)$, but with the following differences on
the way the number of remaining rounds is determined:
\begin{itemize}
\item
Eloise is controlling an ordinal $\gamma_{\exists}$ and Abelard an ordinal $\gamma_{\forall}$.
In the beginning of the game, these ordinals are set to be
equal to $f(\mathcal{M},w,\varphi)$.
\item
Every time a transition is made from some label symbol $X$ to the 
reference formula $\mu X\psi$, Eloise must lower the 
current value of $\gamma_{\exists}$. Similarly, when a 
transition is made from $Y$ to the reference formula $\nu Y \psi'$,
then Abelard must lower $\gamma_{\forall}$. (Note that the
values of $\gamma_{\exists}$ and $\gamma_{\forall}$
are never increased.)
\end{itemize}

If $\gamma_{\exists} = 0$ and we
enter a position where Eloise should lower $\gamma_{\exists}$,
then Eloise loses the game, and
similarly, if $\gamma_{\forall} = 0$ and we enter a position where
Abelard should lower $\gamma_{\forall}$, Abelard loses. In
positions $(\mathcal{M},w',p)$ and $(\mathcal{M},w',\neg p)$ where $p$ is a proposition symbol, winning and losing is
defined in the same way as in $\Gamma$-bounded games.
We define truth of $\varphi$ in $\mathcal{M}$ at $w$
according to the \emph{simple $f$-bounded semantics} 
such that $\mathcal{M},w\Vdash^f\varphi$ iff Eloise has a winning
strategy in the game $\mathcal{G}_f = (\mathcal{M},w,\varphi)$ of
the simple $f$-bounded semantics.
\end{definition}

Henceforth we mostly talk about $f$-bounded semantics instead of simple $f$-bounded
semantics to keep the presentation simpler.

The naturalness and the general 
properties of $f$-bounded semantics of course depend heavily on the choice of $f$.
One of the simpler choices is to define $f\bigl(\mathcal{M},w,\varphi\bigr) = \card(\mathcal{M})
\cdot |\varphi|$ where $|\varphi|$ is the length of $\varphi$, i.e., the number of symbol occurrences.\footnote{Each proposition symbol $p$ and label $X$ counts as one despite the possible
subindices used.} This semantics has the natural property that in
the finite, if the players always lower their ordinal by the
minimum amount $1$, then, if the game
ends due $\gamma_{\exists}$ or $\gamma_{\forall}$ being zero,
then some state-subformula pair must have been repeated.
Furthermore, we can now prove the following result.

\begin{proposition}\label{ptimemodelche}
The  $\mu$-calculus model checking problem is PTime-complete
under simple $f$-bounded semantics
with $f(\mathcal{M},w,\varphi) = \card(\mathcal{M}) \cdot |\varphi|$.
\end{proposition}
\begin{proof}
%
To establish the upper bound, we define a Turing machine running in alternating logarithmic
space that directly simulates the model checking game (i.e.,
the semantic evaluation game) with any input $\mathcal{M},w,\varphi$.
The game positions where Eloise makes a move correspond to existential machine states while Abelard's positions correspond to universal states. We need some kind of a pointer indicating the current world of the Kripke structure 
and another pointer for the current subformula (i.e., node in the
syntax tree). Furthermore, we keep binary
representations of $\gamma_{\exists}$
and $\gamma_{\forall}$ in the memory. These binary strings 
are necessarily logarithmic in the input due to the
choice of $f$. Thus it is easy to see how the
required alternating Turing machine is constructed.

We obtain the lower bound via the alternating reachability game.
Recall Proposition \ref{alternatingreachabilityproposition}
and the formula $\chi$ there.
We will
show that, as
in standard semantics, $\chi$ defines the
winning set of the alternating reachability game also
under our $f$-bounded semantics, i.e., $\chi$
is true in $\mathcal{M}$ at $w$ under our semantics iff
the player $B$ has a
winning strategy in the corresponding alternating reachability game. Indeed, it is
easy to show that when $B$ has a winning
strategy in an alternating reachability game, she can ensure a win so
that no state of the game is visited more than once.
Thus our choice of $f$ for the $f$-bounded
semantics guarantees Eloise has a winning
strategy in the corresponding the semantic game.
And if Eloise has a winning strategy in a semantic game $\mathcal{G}_f(\mathcal{M},w,\chi)$,
then clearly $B$ wins
the corresponding alternating reachability game.
Thus, already with the fixed input formula $\chi$, 
model checking is PTime-hard.
\end{proof}

It is worth
noting here that in fact all
the systems with $f(\mathcal{M},w,\varphi) = \card(\mathcal{M})^k \cdot |\varphi|$
(for different positive integers $k$) have PTime-complete model checking: the proof of Proposition \ref{ptimemodelche} goes through with trivial modifications.

The $f$-bounded semantics with $f(\mathcal{M},w,\varphi) = \card(\mathcal{M}) \cdot |\varphi|$ is
obviously very different in spirit from the standard semantics, and the $f$-bounded semantics itself 
changes as we modify $f$. Also, several further variants of the semantics immediately 
suggest themselves, for example the possibility of setting different limits for Eloise and Abelard, including the possibility of no
limit at all. Also, letting different occurrences
of $\mu$ and $\nu$-formulae be
associated with different clocks similarly to the 
standard semantics, but without resetting the clocks, is one of 
many possible interesting scenarios.

Concerning the case where we do not set clocks at all but allow 
the players to play indefinitely long, winning occurs only 
when an atomic position with a literal (e.g., $p$ or $\neg p$) is reached.
Thus the games are 
not determined, i.e., it is possible that neither player has a winning 
strategy (consider, e.g., the formula $\mu X X$). This \emph{free semantics} for 
modal logic results in a system that is
essentially directly a fragment of the
general, Turing-complete logic $\mathcal{L}$ of \cite{turingcomplete}.
On the other hand, the different `clocking scenarios' described 
above (and further variants thereof) can be naturally imposed on $\mathcal{L}$,
and it would indeed make sense to
study related phenomena in that framework.


\section{Reducing model checking to
alternating \\ reachability}\label{modelchecking}

In this section we study model checking of the $\mu$-calculus for
\emph{fixed sentences}.\footnote{The complexities of the related 
problems are commonly referred to as \emph{data complexity} as 
opposed to the \emph{combined complexity} of the standard
problem where the sentence is not fixed.}
We investigate model checking separately with respect to 
the standard semantics and with respect to $\Gamma$-bounded
semantics, and furthermore, we do not in
general limit to finite models only. Given a sentence
$\varphi$ of the $\mu$-calculus, we use the following notation for the corresponding
model checking problems: 
\begin{itemize}
\item $\mathcal{MC}(\varphi):=\{(\mathcal{M},w)\mid \mathcal{M},w\vDash\varphi\}$, and
\item $\mathcal{BMC}(\varphi):=\{(\mathcal{M},w,\Gamma)\mid 
\mathcal{M},w\vDash^\Gamma\varphi\}$.
\end{itemize}
Furthermore, we denote the restrictions of these classes to finite Kripke models by $\mathcal{MC}_\mathrm{fin}(\varphi)$ and $\mathcal{BMC}_\mathrm{fin}(\varphi)$, respectively.
%
Recalling the relevant notations from Section
\ref{alternatingreachabilitysection}, including the
formula $\chi$, we note, in particular, that $\mathcal{AR}=\mathcal{MC}(\chi)$ and
$\mathcal{AR}_\mathrm{fin}=\mathcal{MC}_\mathrm{fin}(\chi)$.

We will now define natural and easily computable model transformations $J_\varphi$ and $I_\varphi$ such that
for any Kripke model $\mathcal{M}$, state $w$ and ordinal $\Gamma$ we have
\begin{itemize}
\item[(1)]    $(\mathcal{M},w,\Gamma)\in\mathcal{BMC}(\varphi)\;\text{ iff }\;
    J_\varphi(\mathcal{M},w,\Gamma)\in\mathcal{AR}$, and
\item[(2)]    $(\mathcal{M},w)\in\mathcal{MC}(\varphi)\;\text{ iff }\;
    I_\varphi(\mathcal{M},w)\in\mathcal{AR}$.
\end{itemize}
Now, the same equivalences also hold for the finite restrictions
$\mathcal{BMC}_\mathrm{fin}(\varphi)$, $\mathcal{MC}_\mathrm{fin}(\varphi)$
and $\mathcal{AR}_\mathrm{fin}$ of the classes, since $J_\varphi$
and $I_\varphi$ preserve finiteness.
Thus, checking the truth of an arbitrary sentence of the $\mu$-calculus
can be reduced via $I_\varphi$ to checking the truth of the simple
alternation free sentence $\chi$, and this holds both in the
general and in 
the finite case. This result is similar to the ``Measured Collapse Theorem'' in
\cite{vardi}, which also states that checking the truth of any sentence 
$\varphi$ of the $\mu$-calculus can be reduced to checking the truth of 
an alternation free sentence $\varphi'$. However, the result of \cite{vardi}
is not a reduction to $\mathcal{MC}(\psi)$ for a fixed sentence $\psi$;
the sentence $\varphi'$ is actually a translation of $\varphi$ to a different 
logic, called $\mu^\sharp$-calculus, whose semantics is based on an 
additional domain of tuples that can be related to our clock values.

Recall that the game tree of an evaluation
game $\mathcal{G}=(\mathcal{M},w_0,\varphi,\tlimbound)$ is the tree 
$T(\mathcal{G})=(P_\mathcal{G},E_\mathcal{G})$, where $P_\mathcal{G}$ is the set of
positions $(v,\psi,c)$ of $\mathcal{G}$, and $E_\mathcal{G}$ is the successor 
position relation. We consider the following Kripke model that is obtained
 from $T(\mathcal{G})$
by adding proposition symbols encoding winning end positions of Eloise and 
positions in which it is Eloise's turn to move:
$\mathcal{T}_\mathcal{G}=(P_\mathcal{G},E_\mathcal{G},V_\mathcal{G})$,
where $V_\mathcal{G}:\{p_B,q_B\}\to\mathcal{P}(P_\mathcal{G})$
is the valuation
\begin{itemize}
\item $V_\mathcal{G}(p_B)=\{(v,\psi,c)\in P_\mathcal{G}\mid \psi\text{ is a literal and }
\mathcal{M},v\vDash\psi
\}$, 
\item $V_\mathcal{G}(q_B)=\{(v,\psi,c)\in P_\mathcal{G}\mid \psi\text{ is of the form }
\theta\lor\eta,\,\Diamond\theta,\, \mu X\theta, \text{ or $X$ with}$ $\rf(\psi)=\mu X\theta,
\text{ or $\psi$ is a literal and }\mathcal{M},v\nvDash\psi\}$. 
\end{itemize}
Let $r_\mathcal{G}=(w_0,\varphi,c_0)$ be the initial position of $\mathcal{G}$. 
Observe now that, letting Eloise play in the role of $B$ and Abelard in the role of $A$, the alternating
reachability game on the Kripke-model $\mathcal{T}_\mathcal{G}$ with starting state
$r_\mathcal{G}$ is essentially identical with the game $\mathcal{G}$: the positions and the rules for moves are the same, and the winning 
conditions are equivalent.\footnote{For example, in a position $s=(v,\psi,c)$ with $\psi$
a literal such that $\mathcal{M},v\nvDash\psi$, $B$ loses the alternating 
reachability game since $s$ does not have any $E_\mathcal{G}$-successors.}
Thus, defining $J_\varphi(\mathcal{M},w_0,\Gamma):=
(\mathcal{T}_\mathcal{G},r_\mathcal{G})$, and using Theorem 
\ref{the: Equivalence of semantics}, we obtain the first equivalence (1).

The other transformation $I_\varphi$ can now be defined as follows: we let
$I_\varphi(\mathcal{M},w_0):=
J_\varphi(\mathcal{M},w_0,(\card(\mathcal{M}))^+)$. Denote 
$\Gamma^*:=(\card(\mathcal{M}))^+$ below.
By Corollary \ref{the: bounded vs standard semantics} and (1) we have
$$
	(\mathcal{M},w_0)\in\mathcal{MC}(\varphi) \; \text{ iff } \;
	(\mathcal{M},w_0,\Gamma^*)\in\mathcal{BMC}(\varphi) \; \text{ iff } \;
	J_\varphi(\mathcal{M},w_0,\Gamma^*)\in\mathcal{AR},
$$
whence (2) holds.
Furthermore, since $\mathcal{T}_\mathcal{G}$ is finite
whenever the Kripke model $\mathcal{M}$ in $\mathcal{G}$ is finite, $J_\varphi$
and $I_\varphi$ are also reductions from $\mathcal{BMC}_\mathrm{fin}(\varphi)$
and $\mathcal{MC}_\mathrm{fin}(\varphi)$ to $\mathcal{AR}_\mathrm{fin}$ of very
low complexity.\footnote{In fact, assuming that $\Gamma$
is given as an ordered structure, a bisimilar copy of 
$J_\varphi(\mathcal{M},w_0,\Gamma)$ is definable in the input structure 
$(\mathcal{M},w_0,\Gamma)$ by quantifier free first-order formulas.}

It should be noted that the existence of LogSpace-computable reductions from
these model checking problems to $\mathcal{AR}_\mathrm{fin}$ follows
directly from the well-known fact that alternating reachability is a
PTime-complete problem. However, the main point here is that our reductions
$J_\varphi$ and $I_\varphi$ arise in a natural and straightforward way from
the bounded evaluation game, and moreover, they work on infinite Kripke models
as well as on finite ones.



\section{Conclusion and future directions}

Our study has focused on 
conceptual developments relating to the $\mu$-calculus, the 
main result being the new \GTS and its variants.
There are many relevant future research directions; we
mention only a few due to lack of space. 
Firstly, it would be 
interesting to understand 
new \emph{clocking patterns} in general, in addition to the
finitely bounded, the $f$-bounded and the free
semantics discussed above. These investigations could naturally be
pushed to involve more general logics beyond modal logic, possibly 
containing, e.g., operators that modify the underlying models, and
thereby directly linking to the research on the 
general logical framework of \cite{turingcomplete}
and the research program of \cite{turingcomplete} and \cite{final}.

More concretely,
pinpointing the complexity of the satisfiability problem of
the $\mu$-calculus under finitely 
bounded \GTS remains to be done. Also, it would be interesting to
investigate whether the scheme of using 
tuples of ordinals for defining our bounded $\GTS$
can be modified to work with single ordinals in a
natural way. Finally, using ordinals to
reduce arbitrary game arenas to well-founded
trees is in general an interesting 
research direction. Relating to this and the
work in Section \ref{modelchecking}, it would be
particularly interesting to better understand reductions of
general games to
(well-founded) alternating reachability games.


\section*{Appendix}

%

\begin{proof}(Theorem~\ref{the: Equivalence of semantics})
Suppose first that $\mathcal{M},w_0\vDash^\tlimbound\!\varphi_0$. We shall formulate such a strategy for Eloise in the evaluation game $\mathcal{G}=(\mathcal{M},w_0,\varphi_0,\tlimbound)$ that the following condition---called \emph{condition} ($\star$) below---holds in every position $(w,\varphi,c)$ of the game:
%
%
%
%
\begin{itemize}
\item[($\star$)]
There is an assignment $s$ such that $\mathcal{M},w\vDash_s^{\Gamma}\varphi$, and
for each $\X\in\subf(\varphi_0)$:
\begin{enumerate}
\item
$s(X) = (\opM{\psi}_{\X,s,\tlimbound})_\mu^{\gamma}$ if $c(\rf(\X)) = \gamma$
and $\rf(\X) = \mu\X\psi$,
\item
$s(X) = (\opM{\psi}_{\X,s,\tlimbound})_\nu^{\gamma}$ if $c(\rf(\X)) = \gamma$
and $\rf(\X) = \nu\X\psi$.
\end{enumerate}
\end{itemize}
%
%
Note that since we assumed $\varphi_0$ to be in normal form, all different occurrences of a label symbol $\X$ in $\varphi_0$ have the same reference formula. Therefore, in the condition ($\star$), the values $s(\X)$ of each $\X\in\subf(\varphi_0)$ are uniquely defined.
The values $s(Y)$ of label symbols $\Y\in\Lbl\setminus\subf(\varphi_0)$ may be arbitrary.
We then show how Eloise can maintain the condition ($\star$) working inductively
\emph{from the initial position of the game towards ending positions.}
We first observe that the condition ($\star$) holds trivially in the initial position 
since $\mathcal{M},w_0\vDash^\tlimbound\!\varphi_0$ and $\varphi_0$ is a sentence. 
We then establish that in every position $(w,\varphi,c)$ of the game: if ($\star$) holds for $(w,\varphi,c)$,
then Eloise either wins the game or she can \emph{maintain this
condition to the next position of the game.}
%

$\bullet$ Suppose the game is in a position $(w,p,c)$ or $(w,\neg p,c)$. 
If the position is $(w,p,c)$, then by the inductive hypothesis, there is some $s$ such that $\mathcal{M},w\vDash_s^{\Gamma} p$ and thus $w\in V(p)$. Hence Eloise wins the game. The case for the position $(w,\neg p,c)$ is analogous.

$\bullet$ Suppose the game is in $(w,\psi\vee\theta,c)$.  
By the inductive hypothesis, there is some assignment $s$ such that $\mathcal{M},w\vDash_s^{\Gamma}\psi\vee\theta$, i.e., $\mathcal{M},w\vDash_s^{\Gamma}\psi$ or $\mathcal{M}\vDash_s^{\Gamma}\theta$. If the former holds, then Eloise can choose the next position to be $(w,\psi,c)$, and if the latter holds, Eloise can choose the next position to be $(w,\theta,c)$. In both cases ($\star$) holds in the next position of the game.

$\bullet$ Suppose that the game is in $(w,\psi\wedge\theta,c)$.
By the inductive hypothesis, there is some $s$ such that $\mathcal{M},w\vDash_s^{\Gamma}\psi\wedge\theta$, i.e., $\mathcal{M},w\vDash_s^{\Gamma}\psi$ and $\mathcal{M}\vDash_s^{\Gamma}\theta$. Thus ($\star$) holds in both positions $(w,\psi,c)$ and $(w,\theta,c)$. Hence ($\star$) holds in the next position of the game regardless of the
choice of Abelard.

$\bullet$ Suppose that the game is in $(w,\Diamond\psi,c)$.
By the inductive hypothesis, there is some $s$ such that $\mathcal{M},w\vDash_s^{\Gamma}\Diamond\psi$, i.e., there exists some $v\in W$ s.t. $wRv$ and $\mathcal{M},v\vDash_s^{\Gamma}\psi$. Now Eloise can choose the next position to be $(v,\psi,c)$, and the condition ($\star$) holds there.

$\bullet$ Suppose that the game is in $(w,\Box\psi,c)$.
By the inductive hypothesis, there is some $s$ such that $\mathcal{M},w\vDash_s^{\Gamma}\Box\psi$, i.e., $\mathcal{M},v\vDash_s^{\Gamma}\psi$ for every $v\in W$ such that $wRv$. If there is no $v\in W$  such that $wRv$, then Eloise wins the game. Else ($\star$) holds in every possible next position $(v,\psi,c)$ regardless of the choice of Abelard.

$\bullet$ Suppose that the game is in $(w,\mu\X\psi,c)$. 
By the inductive hypothesis, there is some $s'$ such that $\mathcal{M},w\vDash^\tlimbound_{s'}\mu\X\psi$. Therefore there exists some ordinal $\gamma<\tlimbound$ such that $w\in(\opM{\psi}_{\X,s',\tlimbound})_\mu^{\gamma+1}$. Let $A:=(\opM{\psi}_{\X,s',\tlimbound})_\mu^{\gamma}$, whence we have $w\in\opM{\psi}_{\X,s',\tlimbound}(A)$, i.e., $\mathcal{M},w\vDash^{\tlimbound}_{s'[A/\X]}\psi$. Let $s=s'[A/\X]$, whence $s(X)=(\opM{\psi}_{\X,s',\tlimbound})_\mu^{\gamma} = (\opM{\psi}_{\X,s,\tlimbound})_\mu^{\gamma}$ and $s(\Y)=s'(\Y)$ for all $\Y\in\subf(\varphi_0)\setminus\{\X\}$. Now Eloise can choose $\gamma$ as
the clock value of $\rf(\X)$, and therefore the condition ($\star$) holds in the next position $(w,\psi,c[\gamma/\mu\X\psi])$ of the game.

$\bullet$ Suppose that the game is in $(w,\nu\X\psi,c)$.
By the inductive hypothesis, there is some $s'$ such that $\mathcal{M},w\vDash^\tlimbound_{s'}\nu\X\psi$. Therefore $w\in(\opM{\psi}_{\X,s',\tlimbound})_\nu^{\gamma+1}$ for every $\gamma<\tlimbound$. 
Let $\gamma<\tlimbound$ be the clock value of $\rf(\X)$ chosen by Abelard, and let $A:= (\opM{\psi}_{\X,s',\tlimbound})_\nu^{\gamma}$. Now $w\in\opM{\psi}_{\X,s',\tlimbound}(A)$, i.e., $\mathcal{M},w\vDash^{\Gamma}_{s'[A/\X]}\psi$.
Let $s=s'[A/\X]$, whence $s(X)= (\opM{\psi}_{\X,s',\tlimbound})_\nu^{\gamma} = (\opM{\psi}_{\X,s,\tlimbound})_\nu^{\gamma}$ and $s(\Y)=s'(\Y)$ for all $\Y\in\subf(\varphi_0)\setminus\{\X\}$. Hence ($\star$) holds in the next position $(w,\psi,c[\gamma/\nu\X\psi])$.

$\bullet$ Suppose that the game is in $(w,\X,c)$.
We consider two cases, (i) and (ii).

(i) Suppose first that $c(\rf(\X)) = \gamma$ and $\rf(X) = \mu\X\psi$.
By the inductive hypothesis, there is some $s'$ such that $\mathcal{M},w\vDash^\tlimbound_{s'}\X$ and $s'(\X)= (\opM{\psi}_{\X,s',\tlimbound})_\mu^\gamma$. Hence $w\in s'(\X) = (\opM{\psi}_{\X,s',\tlimbound})_\mu^\gamma$, and thus the  clock value $\gamma$ cannot be $0$.
Suppose first that $\gamma$ is a successor ordinal. Let $A:=(\opM{\psi}_{\X,s',\tlimbound})_\mu^{\gamma-1}$, whence we have $w\in \opM{\psi}_{\X,s',\tlimbound}(A)$, i.e., $\mathcal{M},w\vDash^\tlimbound_{s'[A/\X]}\psi$. Let $s=s'[A/\X]$, whence $s(X)= (\opM{\psi}_{\X,s',\tlimbound})_\mu^{\gamma-1} = (\opM{\psi}_{\X,s,\tlimbound})_\mu^{\gamma-1}$ and $s(\Y)=s'(\Y)$ for all $\Y\in\subf(\varphi_0)\setminus\{\X\}$. Now Eloise can lower the clock value of $\rf(\X)$ from $\gamma$ to $\gamma-1$, whence ($\star$) holds in the next position $(w,\psi,c')$.
Suppose then that $\gamma$ is a limit ordinal. Now $w\in\bigcup_{\delta<\gamma}(\opM{\psi}_{\X,s',\tlimbound})_\mu^\delta$, and thus there is some $\delta<\gamma$ s.t. $w\in(\opM{\psi}_{\X,s',\tlimbound})_\mu^{\delta+1}$. Let $A:= (\opM{\psi}_{\X,s',\tlimbound})_\mu^\delta$, whence $w\in \opM{\psi}_{\X,s',\tlimbound}(A)$. Thus Eloise can lower the clock value of $\rf(\X)$ from $\gamma$ to $\delta$, and then ($\star$) holds in the next position of the game by the same reasoning as above.

(ii) Suppose then that $c(\rf(\X))= \gamma$ and $\rf(\X) = \nu\X\psi$.
By the inductive hypothesis, there is some $s'$ such that $\mathcal{M},w\vDash^\tlimbound_{s'}\X$ and $s'(\X)= (\opM{\psi}_{\X,s',\tlimbound})_\nu^\gamma$, and therefore $w\in (\opM{\psi}_{\X,s',\tlimbound})_\nu^\gamma$. If $\gamma=0$, then Eloise wins the evaluation game. Suppose then that $\gamma\neq 0$ and let $\gamma'<\gamma$ be the time limit chosen by Abelard.
Suppose first that the time limit $\gamma$ is a successor ordinal. Since $\gamma'\leq\gamma-1$ and $\opM{\psi}_{\X,s',\tlimbound}$ is monotone, we have $(\opM{\psi}_{\X,s',\tlimbound})_\nu^{\gamma-1} \subseteq (\opM{\psi}_{\X,s',\tlimbound})_\nu^{\gamma'}$.
Let $A:= (\opM{\psi}_{\X,s',\tlimbound})_\nu^{\gamma'}$, whence $w\in \opM{\psi}_{\X,s',\tlimbound}((\opM{\psi}_{\X,s',\tlimbound})_\nu^{\gamma-1}) \subseteq \opM{\psi}_{\X,s',\tlimbound}(A)$, and thus $\mathcal{M},w\vDash^\tlimbound_{s'[A/\X]}\psi$. Let $s=s'[A/\X]$, whence $s(X)=(\opM{\psi}_{\X,s',\tlimbound})_\nu^{\gamma'}=(\opM{\psi}_{\X,s,\tlimbound})_\nu^{\gamma'}$, and thus ($\star$) holds in the next position $(w,\psi,c')$ of the game.
Suppose then that $\gamma$ is a limit ordinal, whence $\gamma'+1<\gamma$. Now $w\in\bigcap_{\delta<\gamma}(\opM{\psi}_{\X,s',\tlimbound})_\mu^\delta$, and thus, in particular, $w\in(\opM{\psi}_{\X,s',\tlimbound})_\mu^{\gamma'+1}$. Let $A:=(\opM{\psi}_{\X,s',\tlimbound})_\mu^{\gamma'}$, whence $w\in \opM{\psi}_{\X,s',\tlimbound}(A)$, and thus ($\star$) holds in the next position by the same reasoning as above.


We have shown that Eloise can maintain the condition ($\star$) at every position until reaching a position where she wins the game. By Proposition~\ref{the: finite games} the game in guaranteed to end in a finite number of rounds, and thus Eloise will eventually win the game by maintaining the condition ($\star$). Hence Eloise has a winning strategy in $\mathcal{G}$, i.e. $\mathcal{M},w_0\Vdash^\tlimbound\!\varphi_0$.

\medskip

We then consider the converse implication of the theorem.
Suppose that $\mathcal{M},w_0\Vdash^\tlimbound\!\varphi_0$, i.e., Eloise has a winning
strategy $\sigma$ in $\mathcal{G}$. 
We next prove by \emph{well-founded induction}\footnote{Note that, by Proposition~\ref{the: finite games}, the game tree of $\mathcal{G}$ is well-founded.} on the game tree of $\mathcal{G}$ that the following
claim holds for every position $(w,\varphi,c)$ in $T(\mathcal{G})$:
\[
	\text{If } (w,\varphi,c) \text{ can be reached with $\sigma$, then ($\star$) holds for } (w,\varphi,c).
\]
It will then follow, in particular, that ($\star$) holds in the initial position of the game and thus we have $\mathcal{M},w_0\vDash^\tlimbound\varphi_0$.

We make the inductive hypothesis that the implication above holds for
every position $(w',\varphi',c')$ that can occur after the position $(w,\varphi,c)$ in the evaluation game $\mathcal{G}$ (that is, there is a path from the node $(w,\varphi,c)$ to the node $(w',\varphi',c')$ in $T(\mathcal{G})$).
Then we prove the implication above for the position $(w,\varphi,c)$.
%

$\bullet$ Suppose that a position $(w,p,c)$ or $(w,\neg p,c)$ can be reached with $\sigma$. 
Suppose first that $(w,p,c)$ can be reached with $\sigma$. Since $\sigma$ is a winning strategy, we must have $w\in V(p)$. Now $\mathcal{M},w\vDash_s^{\Gamma} p$ for any assignment $s$ and thus the condition ($\star$) holds for $(w,p,c)$. The case for the position $(w,\neg p,c)$ is analogous.

$\bullet$ Suppose that $(w,\psi\vee\theta,c)$ can be reached with $\sigma$. 
Let $(w,\xi,c)$, where $\xi\in\{\psi,\theta\}$, be the next position which is chosen according to $\sigma$. By the inductive hypothesis, there is $s$ such that $\mathcal{M},w\vDash_s\xi$. Therefore $\mathcal{M},w\vDash_s^{\Gamma}\psi\vee\theta$ and thus ($\star$) holds for $(w,\psi\vee\theta,c)$.

$\bullet$ Suppose that $(w,\psi\wedge\theta,c)$ can be reached with $\sigma$. 
Now Abelard can choose the next position of the game to be either $(w,\psi,c)$ or $(w,\theta,c)$. Since both of these positions can be reached with $\sigma$, by the inductive hypothesis, there is $s$ such that $\mathcal{M},w\vDash_{s}^{\Gamma}\psi$ and there is $s'$ such that $\mathcal{M},w\vDash_{s'}^{\Gamma}\theta$. 
By the condition ($\star$), $s'$ must have the same values as $s$ for all label symbols occurring in $\varphi_0$, and thus $\mathcal{M},w\vDash_s^{\Gamma}\theta$. Hence $\mathcal{M},w\vDash_s^{\Gamma}\psi\wedge\theta$ and thus ($\star$) holds for $(w,\psi\wedge\theta,c)$.

$\bullet$ Suppose that $(w,\Diamond\psi,c)$ can be reached with Eloise's strategy. 
Let $(v,\psi,c)$, where $v\in W$ s.t. $wRv$, be the next position that is chosen according to $\sigma$. By the inductive hypothesis, there is some $s$ such that $\mathcal{M},v\vDash_s^{\Gamma}\psi$. Therefore $\mathcal{M},w\vDash_s^{\Gamma}\Diamond\psi$, and thus ($\star$) holds
for $(w,\Diamond\psi,c)$.

$\bullet$ Suppose that $(w,\Box\psi,c)$ can be reached with $\sigma$. 
If there is no $v\in W$ such that $wRv$, then $\mathcal{M},w\vDash_s^{\Gamma}\Box\psi$ for any any assignment $s$ and thus the condition ($\star$) holds for $(w,\Box\psi,c)$. Suppose then
that there is some $v'\in W$ such that $wRv'$.
Now Abelard can choose the next position of the game to be $(v,\psi,c)$ for any $v\in W$ s.t. $wRv$. Since all of these positions can be reached with $\sigma$, we observe by the inductive hypothesis that for every $v\in W$ s.t. $wRv$, there is some $s_v$ such
that $\mathcal{M},v\vDash_{s_v}^{\Gamma}\psi$. Define $s:=s_{v'}$. Since all the assignments $s_v$ have the same values for the label symbols of occurring in $\varphi_0$, we have $\mathcal{M},v\vDash_{s}^{\Gamma}\psi$ for
all $v$ such that $wRv$. Therefore $\mathcal{M},w\vDash_{s}^{\Gamma}\Box\psi$ and
thus ($\star$) holds for $(w,\Box\psi,c)$.

$\bullet$ Suppose that $(w,\mu\X\psi,c)$ can be reached with $\sigma$. 
Let $\gamma<\tlimbound$ the clock value that is chosen by $\sigma$, whence the next
position of the game is $(w,\psi,c[\gamma/\mu\X\psi])$. By the inductive hypothesis, there
is some $s$ such that $\mathcal{M},w\vDash^\tlimbound_s\psi$ and $s(X)=(\opM{\psi}_{\X,s,\tlimbound})_\mu^\gamma$. 
Hence it holds that $w\in\opM{\psi}_{\X,s,\tlimbound}(s(\X))=\opM{\psi}_{\X,s,\tlimbound}((\opM{\psi}_{\X,s,\tlimbound})_\mu^\gamma)=(\opM{\psi}_{\X,s,\tlimbound})_\mu^{\gamma+1}$, and thus $\mathcal{M},w\vDash_s^{\Gamma}\mu\X\psi$. The assignment $s$ satisfies the requirements of ($\star$) for the position $(w,\mu\X\psi,c)$. Thus the condition ($\star$) holds for $(w,\mu\X\psi,c)$.

$\bullet$ Suppose that $(w,\nu\X\psi,c)$ can be reached with $\sigma$.
Since Abelard may choose any $\gamma<\tlimbound$ as the clock value, the next position of the game can be $(w,\psi,c[\gamma/\nu\X\psi])$ for any $\gamma<\tlimbound$. All of these positions can be reached with $\sigma$. Hence, by the inductive hypothesis, for every $\gamma<\tlimbound$, there is some $s_\gamma$ such that $\mathcal{M},w\vDash^\tlimbound_{s_\gamma}\psi$ and $s_\gamma(X)=(\opM{\psi}_{\X,s_\gamma,\tlimbound})_\nu^\gamma$. 
Note that all the assignments $s_{\gamma}$ (for different values $\gamma<\Gamma$) agree on all other label symbols in $\varphi_0$ except $\X$.
Define $s:=s_0$ and let $\gamma<\tlimbound$.
Now $w\in\opM{\psi}_{\X,s_\gamma,\tlimbound}(s_\gamma(\X))=\opM{\psi}_{\X,s_\gamma,\tlimbound}((\opM{\psi}_{\X,s_{\gamma},\tlimbound})_\nu^\gamma)=(\opM{\psi}_{\X,s_\gamma,\tlimbound})_\nu^{\gamma+1}=(\opM{\psi}_{\X,s,\tlimbound})_\nu^{\gamma+1}$. Since this holds for any $\gamma<\tlimbound$, we have $\mathcal{M},w\vDash_s^{\Gamma}\nu\X\psi$. The assignment $s$ satisfies the requirements  of ($\star$) and thus the condition ($\star$) holds for $(w,\nu\X\psi,c)$.

$\bullet$ Suppose that $(w,\X,c)$ can be reached with $\sigma$. 

%
(i) Suppose first that $\rf(\X)=\mu\X\psi$ and $c(\rf(\X))=\gamma$. Since $\sigma$ is a winning strategy for Eloise, we must have $\gamma\neq 0$. Let $\gamma'<\gamma$ be the clock value chosen according to $\sigma$, whence the next position of the game is $(w,\psi,c')$ where $c'(\mu X\psi) = \gamma'$.
Therefore we observe by the inductive hypothesis that there is a suitable $s'$ such that $\mathcal{M},w\vDash^\tlimbound_{s'}\psi$ and $s'(\X)=(\opM{\psi}_{\X,s',\tlimbound})_\mu^{\gamma'}$.
Suppose first that the time limit $\gamma$ is a successor ordinal. Let $A:=(\opM{\psi}_{\X,s',\tlimbound})_\mu^{\gamma}$, whence $A=\opM{\psi}_{\X,s',\tlimbound}((\opM{\psi}_{\X,s',\tlimbound})_\mu^{\gamma-1})$. Since $\gamma'<\gamma$, we have $\gamma'\leq\gamma-1$. 
As $\mathcal{M},w\vDash^\tlimbound_{s'}\psi$, we have $w\in \opM{\psi}_{\X,s',\tlimbound}(s'(\X))$. Thus $w\in \opM{\psi}_{\X,s',\tlimbound}((\opM{\psi}_{\X,s',\tlimbound})_\mu^{\gamma'}) \subseteq \opM{\psi}_{\X,s',\tlimbound}((\opM{\psi}_{\X,s',\tlimbound})_\mu^{\gamma-1}) = A$. Let $s=s'[A/\X]$, whence $w\in A = s(\X)$ and thus $\mathcal{M},w\vDash^\tlimbound_{s} \X$. Now $s(\X)=(\opM{\psi}_{\X,s',\tlimbound})_\mu^{\gamma}=(\opM{\psi}_{\X,s,\tlimbound})_\mu^{\gamma}$ and $s(\Y)=s'(\Y)$ for all $\Y\in\subf(\varphi_0)\setminus\{\X\}$. Therefore ($\star$) holds  for $(w,\X,c)$.
Suppose then that $\gamma$ is a limit ordinal. Let $A:=(\opM{\psi}_{\X,s',\tlimbound})_\mu^{\gamma}$, whence we have $A=\bigcup_{\delta<\gamma}(\opM{\psi}_{\X,s',\tlimbound})_\mu^{\delta}$. Since $\gamma'<\gamma$ and $\gamma$ is a limit ordinal, $\gamma' +1 <\gamma$.
As $\mathcal{M},w\vDash^\tlimbound_{s'}\psi$, we have $w\in \opM{\psi}_{\X,s',\tlimbound}(s'(\X))$ and thus $w\in \opM{\psi}_{\X,s',\tlimbound}((\opM{\psi}_{\X,s',\tlimbound})_\mu^{\gamma'})= (\opM{\psi}_{\X,s',\tlimbound})_\mu^{\gamma' + 1}\subseteq A$. Let now $s:=s'[A/\X]$, whence ($\star$) holds for $(w,\X,c)$ by similar reasoning as above.

(ii) Suppose then that $\rf(\X)=\nu\X\psi$ and $c(\rf(\X))=\gamma$. Suppose first that $\gamma=0$, and let $s$ be an assignment whose values satisfy the requirements of ($\star$) with respect to the values of $c$. Now, in particular, $s(\X)=(\opM{\psi}_{\X,s,\tlimbound})_\nu^0=W$ and thus trivially $\mathcal{M},w\vDash^\tlimbound_s\X$. Hence the condition ($\star$) holds for $(w,\X,c)$.
Suppose then that $\gamma>0$. Abelard may now choose any $\gamma'<\gamma$, whence the next position of the game is $(w,\psi,c_{\gamma'})$, where $c_{\gamma'}(\nu\X\psi)=\gamma'$. All such positions can be reached with $\sigma$. Hence, by the inductive hypothesis, for every $\gamma'<\gamma$, there is a suitable $s_{\gamma'}$ such that $\mathcal{M},w\vDash^\tlimbound_{s_{\gamma'}}\psi$ and $s_{\gamma'}(\X)=(\opM{\psi}_{\X,s_{\gamma'},\tlimbound})_\nu^{\gamma'}$.
Suppose first that $\gamma$ is a successor ordinal. Let $s':=s_{\gamma-1}$, whence $\mathcal{M},w\vDash^\tlimbound_{s'}\psi$ and $s'(\X)=(\opM{\psi}_{\X,s',\tlimbound})_\nu^{\gamma-1}$.  Let $A:=(\opM{\psi}_{\X,s',\tlimbound})_\nu^{\gamma}$. As $\mathcal{M},w\vDash^\tlimbound_{s'}\psi$, we have $w\in\opM{\psi}_{\X,s',\tlimbound}(s'(\X))=\opM{\psi}_{\X,s',\tlimbound}((\opM{\psi}_{\X,s',\tlimbound})_\nu^{\gamma-1})=A$.
Let $s=s'[A/\X]$, whence $w\in A = s(\X)$ and thus $\mathcal{M},w\vDash^\tlimbound_{s}\X$. Now $s(\X)=(\opM{\psi}_{\X,s',\tlimbound})_\nu^{\gamma}=(\opM{\psi}_{\X,s,\tlimbound})_\nu^{\gamma}$ and $s(\Y)=s'(\Y)$ for all $\Y\in\subf(\varphi_0)\setminus\{\X\}$. Therefore ($\star$) holds for $(w,\X,c)$. 
Suppose then that $\gamma$ is a limit ordinal. Let $s_0$ be the assignment corresponding to Abelard's choice $\gamma'=0$. Let $A:=(\opM{\psi}_{\X,s_0,\tlimbound})_\nu^{\gamma}$, whence  $A=\bigcap_{\delta<\gamma}(\opM{\psi}_{\X,s_0,\tlimbound})_\nu^{\delta}$. For the sake of proving that $w\in A$, let $\delta<\gamma$. Now there is some suitable $s_\delta$ such that $\mathcal{M},w\vDash^\tlimbound_{s_\delta}\psi$ and $s_\delta(\X)=(\opM{\psi}_{\X,s_\delta,\tlimbound})_\nu^{\delta}$.
Note that $s_0$ and $s_{\delta}$ agree on all other label symbols occurring in $\varphi_0$ except $X$.
As $\mathcal{M},w\vDash^\tlimbound_{s_\delta}\psi$, it holds that $w\in\opM{\psi}_{\X,s_\delta,\tlimbound}(s_\delta(\X))=\opM{\psi}_{\X,s_\delta,\tlimbound}((\opM{\psi}_{\X,s_\delta,\tlimbound})_\nu^{\delta})=(\opM{\psi}_{\X,s_\delta,\tlimbound})_\nu^{\delta+1} \subseteq (\opM{\psi}_{\X,s_\delta,\tlimbound})_\nu^{\delta} = (\opM{\psi}_{\X,s_0,\tlimbound})_\nu^{\delta}$. Since this holds for every $\delta<\gamma$, we have $w\in A$. Let now $s:=s_0[A/\X]$, whence ($\star$) holds for $(w,\X,c)$ by similar reasoning as above.
\end{proof}

\bibliographystyle{plain}
\bibliography{mucalc}

\end{document}